\documentclass[a4paper,reqno]{amsart}
\usepackage{amsmath, amssymb, amsthm, epsfig, latexsym,enumerate}
\usepackage[colorlinks=true, linkcolor=blue, citecolor=red]{hyperref}

\usepackage{url}
\usepackage[mathscr]{euscript}
\usepackage{color}
\usepackage{setspace}

\usepackage{refcheck} 
\norefnames
\nocitenames

\def\today{\ifcase\month\or
  January\or February\or March\or April\or May\or June\or
  July\or August\or September\or October\or November\or December\fi
  \space\number\day, \number\year}
\DeclareMathOperator{\sgn}{\mathrm{sgn}}

 \newtheorem{theorem}{Theorem}
  \newtheorem{conjecture}[theorem]{Conjecture}
 \newtheorem{lemma}[theorem]{Lemma}
 \newtheorem{proposition}[theorem]{Proposition}
 \newtheorem{corollary}[theorem]{Corollary}
 
  \newtheorem{question}[theorem]{Question}

 \theoremstyle{definition}

 \theoremstyle{remark}

 \newcommand{\ft}{\widehat}
 \newcommand{\mc}{\mathcal}
 \newcommand{\A}{\mc{A}}

 \newcommand{\F}{\mc{F}}

 \newcommand{\I}{\mc{I}}

 \newcommand{\LL}{\mc{L}}

 \newcommand{\C}{\mathbb{C}}
 \newcommand{\R}{\mathbb{R}}
  \newcommand{\Rn}{\mathbb{R}^N}

 \newcommand{\Z}{\mathbb{Z}}
 
 \newcommand{\tF}{\widehat{F}}

  \newcommand{\bu}{\boldsymbol{u}}
 
  \newcommand{\bn}{\boldsymbol{n}}
  \newcommand{\bm}{\boldsymbol{m}}
 \newcommand{\bx}{\boldsymbol{x}}
  \newcommand{\bl}{\boldsymbol{\ell}}
  
  \newcommand{\by}{\boldsymbol{y}}
     
          \newcommand{\bv}{\boldsymbol{v}}
   \newcommand{\bt}{\boldsymbol{t}}

\newcommand{\es}[1]{\begin{align}#1\end{align}}
\newcommand{\est}[1]{\begin{align*}#1\end{align*}}

\newcommand{\wt}{\widetilde}

\newcommand{\ep}{\varepsilon}

\newcommand{\bo}{\boldsymbol}

\newcommand{\bxi}{\boldsymbol{\xi}}
\newcommand{\dint}{\displaystyle\int}
\newcommand{\dprod}{\displaystyle\prod}
\newcommand{\dsum}{\displaystyle\sum}
\newcommand{\ra}{\rightarrow}

\date{\today}
\begin{document}

\title[The Box-Minorant Problem]{The Beurling-Selberg Box Minorant Problem via Linear Programming Bounds}
\author[Carruth]{Jacob Carruth}
\address{Univ. of Texas, 1 University Sta. Austin, TX 78712}
\email{jcarruth@math.utexas.edu}
\author[Elkies]{Noam Elkies}
\address{Harvard University, Cambridge, MA 02138}
\email{elkies@math.harvard.edu}
\author[Gon\c{c}alves]{Felipe Gon\c{c}alves}
\address{Hausdorff Center for Mathematics, Universit\"at Bonn, Endenicher Allee 60, 53115 Bonn, Germany}
\email{goncalve@math.uni-bonn.de}
\author[Kelly]{Michael Kelly}
\address{Center for Communications Research, 805 Bunn Dr. Princeton, NJ 08540}
\email{mskelly@idaccr.org}
%
\allowdisplaybreaks
\numberwithin{equation}{section}

\begin{abstract}
In this paper we investigate a high dimensional version of Selberg's minorant problem for the indicator function of an interval. In particular, we study the corresponding problem of minorizing the indicator function of the box $Q_{N}=[-1,1]^N$ by a function whose Fourier transform is supported in the same box $Q_N$. We show that when the dimension is sufficiently large there are no minorants with positive mass and we give an explicit lower bound for such dimension. On the other hand, we explicitly construct minorants for dimensions $1,2,3,4$ and $5$ and, as an application, we use them to produce an improved diophantine inequality for exponential sums.
\end{abstract}
 
\maketitle

\section{Introduction}
Let $Q_N=[-1,1]^N$, let ${\bo 1}_{Q_{N}}$ be the indicator function of $Q_{N}$, and let $\delta>0$. A fundamental question in approximation theory asks:

\begin{question}\label{mainQuestion} 
Does there exist a function $F:\R^N\to \R$ such that:
\begin{enumerate}[$(i)$]
\item $F(\bx)\leq {\bf 1}_{Q_{N}}(\bx)$ for all $\bx\in\R^N$;
\item the Fourier transform of $F(\bx)$ is supported in the box $[-\delta,\delta]^N$;
\item $\int_{\R^N}  F(\bx)d\bx>0$?
\end{enumerate}
\end{question}

Note that condition $(iii)$ is a natural one, it only imposes that $F(\bx)$ does a better job  than the trivial minorant $F\equiv 0$. Basic considerations will lead the reader to surmise that the existence of such a function depends on the size of $\delta$.  If $\delta$ is very large, then such a function will surely exist. On the other hand, if $\delta$ is very small, then no such function ought to exist. When  $N=1$ the above question was settled by Selberg \cite{S,V} who showed that there is a positive answer to Question \ref{mainQuestion} if and only if $\delta>\tfrac{1}{2}$. From here it is not difficult to show that when $N>1$, Question \ref{mainQuestion} has a negative answer whenever $\delta\leq\tfrac{1}{2}$ (see Lemma \ref{slice}). When $N$ is large it is unknown how small $\delta$ may be for Question \ref{mainQuestion} to admit a positive answer. The best result in this direction is due to Selberg who proved that when $N>1$ and $\delta>N-\tfrac{1}{2}$, then Question \ref{mainQuestion} has a positive answer. Selberg never published his construction, but he did communicate it to Vaaler and Mongtomery (personal communication). His construction has since appeared several times in the literature, see for instance \cite{Harman1993,Harman1998, HKW}. More details about Selberg's (and also Montgomery's) construction can be found in the Appendix.
\smallskip

The following is the main theorem of this paper.

\begin{theorem}
 If $N>717$ and $\delta= 1$ then Question \ref{mainQuestion} has a negative answer. In contrast, Question \ref{mainQuestion} has a positive answer for $\delta=1$ in dimension $N=1,2,3,4,5$. 
\end{theorem}

The proof of this, as well as and our other main results, are based in a detailed analysis of the following extremal problem and a novel technique to bound the objective of this infinite dimensional linear program by the objective of a finite dimensional linear program (Theorems \ref{nuLPbound} and \ref{DeltaLPbound}).

\begin{question}\label{prob2} 
For every integer $N\geq 1$ determine the value of the quantity
\begin{equation}\label{nu}
		\nu(N)=\sup\dint_{\Rn}F(\bx)d\bx,
	\end{equation}
where the supremum is taken over functions $F\in L^1(\R^N)$ such that:

\begin{center}
\begin{enumerate}[$(I)$]
\item $\ft F(\bxi)$ is supported in $Q_N$;
\item $F(\bx)\leq {\bo 1}_{Q_{N}}(\bx)$ for (almost) every $\bx\in\Rn$.
\end{enumerate}
\end{center}
\end{question}

We show that the admissible minorants are given by a Whittaker-Shannon type interpolation formula and we use this formula to demonstrate that the {only admissible minorant with non-negative integral that interpolates the indicator function ${\bo 1}_{Q_N}(\bx)$ at the integer lattice $\Z^N\setminus\{{\bo 0}\}$ is the { identically zero function}. We also define an auxiliary quantity $\Delta(N)$ in \eqref{delta-quantity}, similar to $\nu(N)$, and derive a functional inequality, which ultimately implies that $\nu(N)$ {vanishes} for finite $N$.}
\smallskip

As we have remarked,  when $N=1$ Question \ref{mainQuestion} is settled completely. In fact, even the corresponding { extremal} problem is completely settled (in higher dimensions no extremal results for the box minorant problem are known). Suppose $\I$ is an interval in $\R$ of finite length, ${\bo 1}_{\I}(x)$ is the indicator of $x$, and  $\delta>0$.  Selberg \cite{S,V} introduced functions $C(x)$ and $c(x)$ with the following properties:
\smallskip
	\begin{enumerate}[$(i)$]
		\item $\ft{C}(\xi)=\ft{c}(\xi)=0$ if $|\xi|>\delta$ (where $ \ \ft{\nonumber} \ $ denotes the Fourier transform);
		\item $c(x)\leq {\bo 1}_{\I}(x)\leq C(x)$ for each $x\in\R,$;
		\item $\int_{-\infty}^{\infty} (C(x)-{\bo 1}_{\I}(x))dx =\int_{-\infty}^{\infty}({\bo 1}_{\I}(x)-c(x))dx =\delta^{-1}$.
	\end{enumerate}
	\smallskip
Furthermore, among all functions that satisfy $(i)$ and $(ii)$ above, Selberg's functions minimize the integrals appearing in $(iii)$ if and only if $\delta\,\mathrm{length}(\I)\in\Z$. If  $\delta\,\mathrm{length}(\I)\not\in\Z$, then the extremal functions have been found by Littmann in \cite{Littquad}. 
\smallskip

Originally, Selberg was motivated to construct his one dimensional extremal functions to prove a sharp form of the large sieve. His functions and their generalizations have since become part of the standard arsenal in analytic number theory and have a number of applications in fields ranging from probability, dynamical systems, optics, combinatorics, sampling theory, and beyond. For a non-exhaustive list see
 \cite{MR2739041,MR3063902,MR3384872,MR3343896,MR3078273,MR3110588,MR3209354,MR3042593,MR2661497,MR2581230,MR2781205,MR2331578,MR1722198,MR0466048} and the references therein. 
\smallskip
	
In recent years higher dimensional analogues of Selberg's extremal function and related constructions have proven to be important in the recent studies of Diophantine inequalities \cite{BMV,MR947641,HKW,HV}, visibility problems and {quasicrystals \cite{Adiceam, HKW}}, and sphere packings\footnote{The extremal problems considered for sphere packings differ from the problems that we consider here. Instead of the admissible functions being band-limited, their Fourier transforms are only required to be non-negative.}
\cite{2016arXiv160306518C,cohn2002new,cohn2003new,cohn2009optimality,cohn2014sphere,2016arXiv160304246V}. See also \cite{MR3193963} for related constructions recently used in signal processing. Since Selberg's original construction of his box minorants there has been some progress on the Beurling-Selberg problem in higher dimensions \cite{BMV,BK,CL3,CL4,GKM2014}. In particular, in \cite{HV} Holt and Vaaler initiate the study of a variant of Question \ref{mainQuestion} in which the boxes are replaced by Euclidean balls. They are actually able to establish extremal results in some cases. A complete solution to Question \ref{mainQuestion} for balls can be found in \cite{G}. There seems to be a consensus among experts that despite four decades of progress on Beurling-Selberg problems, box minorants are poorly understood. This sentiment was recently raised in \cite{MR3479169}. We hope that the contributions of this paper will help reveal why the box minorant problem is so difficult and move us closer to understanding these enigmatic objects.

\section{Main Results}
In this section we give some definitions and state the main results of the present article. A function $F(\bx)$ satisfying conditions (I) and (II) of Question \ref{prob2} will be called { admissible for} $\nu(N)$ (or $\nu(N)$-admissible) and if it achieves equality in \eqref{nu}, then it is said to be {extremal}.

An indispensable tool in our investigation is the Poisson summation formula. If $G:\R^N \to \R$ is ``sufficiently nice'' (see \cite{SW} for a precise statement of when the formula holds), $\Lambda$ is a full rank lattice in $\R^N$ of covolume $|\Lambda|$, and $\Lambda^{*}$ is the corresponding dual lattice\footnote{That is, $\Lambda^{*}=\{ {\bo u}\in \R^N : {\bo u}\cdot \boldsymbol{\lambda}\in \Z \text{ for all }\boldsymbol{\lambda}\in \Lambda  \}$.}, then the Poisson summation formula is the assertion that 
\begin{equation} \label{poisson}
   \sum_{\boldsymbol{\lambda}\in \Lambda} G({\bo x}+\lambda) = \frac{1}{|\Lambda|}\sum_{{\bo u}\in\Lambda^*} \ft{G}({\bo
   u})e^{2\pi i {\bo u}\cdot {\bo x}},
\end{equation}
for every ${\bo x}\in \R^N$.

If $F(\bx)$ is a $\nu(N)$-admissible function, then it follows from Proposition \ref{prop:Poisson} that the Poisson summation formula may be applied to $F({\bo x})$. That is, $\nu(N)$-admissible functions are ``sufficiently nice.'' Thus, upon applying Poisson summation  \eqref{poisson-sum} to $F({\bo x})$ we find that
	\[
		\tF({\bo 0})= \dsum_{\bn\in\Z^N} \tF(\bn) = \dsum_{\bn\in\Z^N} F(\bn) \leq F({\bo 0}).
	\]
Thus we have the fundamental inequality 
	\begin{equation}\label{fundamental}
		\tF({\bo 0}) \leq F({\bo 0}).
	\end{equation}
Evidently there is equality in \eqref{fundamental} if, and only if, $F(\bn)=0$ for each non-zero $\bn\in\Z^N$. If $N=1$ then, by using the interpolation formula \eqref{int-form-gen}, Selberg was able to show (see \cite{S,V}) that $\nu(1)=1$ and that
	\begin{equation*}
		\dfrac{\sin^{2}\pi x}{(\pi x)^{2}(1-x^2)}
	\end{equation*}
is an extremal function (this is not the unique extremal function). We also note that the Fourier transform of the above function is non-negative, supported in $|x|\leq 1$ and equal to
$$
1-|x|+\frac{\sin(2\pi |x|)}{2\pi}.
$$
Therefore, Selberg's function is also extremal for the Cohn and Elkies \cite{cohn2003new} linear programming bounds for sphere packings in dimension $1$ (again not unique).

A more refined version of the inequality \eqref{fundamental} can be obtained by taking a weighted average of the Poisson summation formula on grids. More precisely, suppose that $\Lambda\subset \R^N$  is a full-rank lattice,  $y_{1},...,y_{L}\in \R^N$, and $\omega_{1},...,\omega_{L}\geq 0$. By repeatedly applying \eqref{poisson} and interchanging the order of summation, we find that
\begin{equation}\label{weightedPoisson}
\sum_{\ell=1}^{L}\omega_{\ell}\sum_{\bo \lambda\in\Lambda}F(\bo \lambda+\by_{\ell}) = \frac{1}{|\Lambda|}\sum_{\bu\in\Lambda^*}\ft {F}(\bu)\sum_{\ell=1}^{L}\omega_{\ell}e^{-2\pi i \bu\cdot \by_{\ell}}.
\end{equation}

\noindent Suppose that 
\begin{equation}\label{lp1}
\sum_{\ell=1}^{L}\omega_{\ell}e^{-2\pi i \bu\cdot \by_{\ell}} = 0, \ \;\text{if } \bo u \in \Lambda^* \cap Q_N\setminus \{\bo 0\},
\end{equation}
and 
\begin{equation}\label{lp2}
\sum_{\ell=1}^{L}\omega_{\ell}=|\Lambda|.
\end{equation}
If $F(\bx)$ is $\nu(N)$-admissible, then \eqref{weightedPoisson} yields the following strengthening of \eqref{fundamental}:

\begin{equation}\label{fundamentalLP}
   \ft{F}({\bo 0})\leq \underset{\|\bo\lambda+\by_{\ell}\|_\infty<1}{\sum_{\ell=1}^{L}\sum_{\bo \lambda\in \Lambda}}\omega_{\ell}.
\end{equation}
Since the right hand side of \eqref{fundamentalLP} is a finite sum that is linear in $\omega_{1},...,\omega_{L}$, we have the following { finite dimensional} linear programming bounds for $\nu(N)$.

\begin{theorem}\label{nuLPbound}
 Suppose $\bo y_{1},...,\bo y_{L}\in \R^N$ and that $\Lambda$ is a full rank lattice in $\R^N$ of covolume $|\Lambda|$. Then
\[
   \nu(N) \leq \min \underset{\|\bo \lambda+\bo y_{\ell}\|_\infty<1}{\sum_{\ell=1}^{L}\sum_{\lambda\in \Lambda}}\omega_{\ell}
\]
where the minimum is taken over $\omega_{1},...,\omega_{L}\geq 0$ satisfying \eqref{lp1} and \eqref{lp2}.
\end{theorem}

We believe that the above result should actually give the optimal answer.
\begin{conjecture}
Assume $\nu(N)>0$. Then for any $\ep>0$ there exists a full rank lattice $\Lambda\subset \R^N$, vectors $\bo y_{1},...,\bo y_{L}\in \R^N$ and numbers $\omega_{1},...,\omega_{L}\geq 0$ satisfying \eqref{lp1} and \eqref{lp2} such that
$$
\nu(N)+\ep > \underset{\|\bo \lambda+\bo y_{\ell}\|_\infty<1}{\sum_{\ell=1}^{L}\sum_{\lambda\in \Lambda}}\omega_{\ell}.
$$
\end{conjecture}
The following theorem compiles some of the basic properties related to the quantity $\nu(N)$, establishing: $(1)$ that extremizers for the quantity $\nu(N)$ do exist, $(2)$ that $\nu(N)$ is a decreasing function of $N$ and, most curiously, $(3)$ that $\nu(N)$ { vanishes for finite $N$}.

\begin{theorem}\label{nuTheorem}
The following statements hold.
\begin{enumerate}[$(i)$]
\item For every $N\geq 2$ there exists a $\nu(N)$-admissible function $F(\bx)$ such that 
$$
\nu(N)=\int_{\R^N} F(\bx)d\bx.
$$
\item If $\nu(N)>0$ then $\nu(N+1) < \nu(N)$. In particular, $\nu(2)<1$.
\item There exists a critical dimension $N_c$ such that $\nu(N_c)>0$ and $\nu(N)=0$ for all $N> N_c$. Moreover,
\est{
5\leq N_c \leq \bigg\lfloor\frac{k}{1-\Delta(k)}\bigg\rfloor.
}
for any $k\leq N_c$.
\end{enumerate}
\end{theorem}

\noindent {\bf Remarks.} 
\begin{enumerate}[$(1)$]
\item Using Theorem \ref{DeltaLPbound} we were able to show that $\Delta(2) < .997212$, yielding an upperbound $N_c\leq 717$ (see Table \ref{table:data}).
 \smallskip
 
\item The quantity $\Delta(k)$ appearing in the above theorem is defined in equation \eqref{delta-quantity}. It follows from Lemma \ref{Delta-dec-lemma} that $k\mapsto k/(1-\Delta(k))$ is non-increasing for $k\leq N_c$, and from Theorem \ref{Deltathm} that $\Delta(k)<1$ for all $k\geq 2$. Thus producing upper bounds for $\Delta(k)$ in higher dimensions will improve the critical dimension $N_c$, however the problem quickly becomes incredibly hard as the dimension increases, demanding a huge amount of computational time to deliver an upper bound strictly less than one. That is why we were only able to produce upper bounds up to dimension 5. Moreover, the above result can only be applied for $k\leq N_c$ and so far we do not known if $\nu(6)>0$, thus to use the upper bound derived above in a dimension higher than $5$, we have also to find a non-trivial minorant in such dimension.
\smallskip
 
\item To put this result in context, note that volume of $Q_N $ is growing exponentially, so there is a lot of volume on both the physical and frequency sides. However, every time another dimension gets added, more constraints also get added so it requires a detailed analysis to determine the behavior of $\nu(N)$. Poisson summation, which yields the non-intuitive bound $\nu(N)\leq 1$, already detects this tug-of-war.
\smallskip

\item Theorem \ref{nuTheorem} has some parallels in classical asymptotic geometric analysis, and mass concentration in particular. In our first attempts to prove Theorem \ref{nuTheorem} we tried to employ asymptotic geometric techniques to exploit properties of $Q_N$ but we were not able to uncover a proof. We found it awkward to incorporate the Fourier analytic and one-sided inequality constraints (i.e. (I) and (II) in the definition of $\nu(N)$) with the standard tool kit of asymptotic geometric analysis. It would be very interesting to see a proof of Theorem \ref{nuTheorem} based on such techniques. 
\end{enumerate}
\smallskip

Our next result shows that Selberg's $\Z^N$-interpolation strategy to build minorants fails in higher dimensions.

\begin{theorem}\label{latticeInterpolateN}
Let $N\geq 2$. Let $F(\bx)$ be an admissible function for $\nu(N)$ and assume that $F({\bo 0})\geq 0$. If $F(\bn)=0$ for every $\bn\in\Z^N\setminus\{\bo 0\}$, then $F(\bx)$ vanishes identically.
\end{theorem}

We are also interested in studying a ``scaled-out'' version of the $\nu(N)$-problem defined as follows. Let 
\es{\label{delta-quantity}
\Delta(N)=\sup_{F} \int_{\R^N} F(\bx)d\bx,
}
where the supremum if taken among functions $F(\bx)$ such that:
\begin{center}
\begin{enumerate}[(I)]
\item $\ft F(\bxi)$ is supported in $Q_N$;
\item $F(\bx)\leq 0$ for (almost) every $\bx\notin Q_N$;
\item $F(\bo 0)=1$;
\item $\ft F(\bo 0)> 0$.
\end{enumerate}
\end{center}
We have the following analogue of Theorem \ref{nuLPbound} for $\Delta(N)$.

\begin{theorem}\label{DeltaLPbound}
 Suppose $y_{1}={\bo 0}$,  $y_{2},...,y_{L}\in \R^N$ and that $\Lambda$ is a full rank lattice in $\R^N$ of covolume $|\Lambda|$. Then
\[
   \Delta(N) \leq \min \omega_{1}
\]
where the minimum is taken over $\omega_{1},...,\omega_{L}\geq 0$ satisfying \eqref{lp1}, \eqref{lp2}, and for $\ell=2,..,L$ 
\[
 \omega_{\ell}=0 \text{ if }\|\bo \lambda+y_{\ell}\|_\infty<1 \text{ for some } \bo\lambda\in\Lambda.
\]
\end{theorem}

The quantity $\Delta(N)$ may not be well defined for some $N$, in this case we define $\Delta(N)=0$. Lemma \ref{slice} shows that if $\Delta(N_0)$ is well-defined (that is $\Delta(N_0)>0$), then it is well defined for all $N\leq N_0$. One can also verify that $\Delta(N)>0$ if and only if $\nu(N)>0$ and 
$$
\nu(N)\leq \Delta(N).
$$ 
Thus, they vanish for the first time at the same dimension. Poisson summation shows that $\Delta(N)\leq 1$ for all $N$ and thus $\Delta(1)=1$. A priori, the existence of extremizers for the $\Delta(N)$ problem is not guaranteed since an extremizing sequence may blow-up inside the box $Q_N$. The next theorem shows that $\Delta(N)$ behaves similarly to $\nu(N)$ for $N\geq 2$.

\begin{theorem}\label{Deltathm}
The following statements hold.
\begin{enumerate}[$(i)$]
\item There exists a constant $B_N\geq 1$, depending only on $N$, such that if $F(\bx)$ is admissible for the $\Delta(N)$ problem then $F(\bx)\leq B_N$ for all $x\in Q_N$.
\item If $\Delta(N)>0$, then there exists a $\Delta(N)$-admissible function $F(\bx)$ such that 
$$
\Delta(N)=\int_{\R^N} F(\bx)d\bx.
$$
\item If $\Delta(N)>0$, then $\Delta(N+1) < \Delta(N)$. In particular, $\Delta(2)<1$.
\item There exists a critical dimension $N_c$ such that $\Delta(N_c)>0$ and $\Delta(N)=0$ for all $N>N_c$. Moreover, the same bound holds
\est{
5\leq N_c \leq \bigg\lfloor\frac{k}{1-\Delta(k)}\bigg\rfloor,
}
for any $k\leq N_c$.
\end{enumerate}
\end{theorem}

We now give some explicit lower bounds for the quantity $\nu(N)$ up to dimension $N=5$ (see Theorem \eqref{construction-lower-bound}). These are constructed explicitly in Section \ref{conMin}.

\begin{theorem}
	 We have the following lower bounds for $\nu(N)$:
		\begin{itemize}
			\item $\nu(2)\geq \tfrac{63}{64} =0.984375 ,$
			\item  $\nu(3)\geq \tfrac{119}{128} = 0.9296875,$
			\item  $\nu(4)\geq \tfrac{95}{128}= 0.7421875,$
			\item $\nu(5)\geq \tfrac{31}{256}=0.12109375$.
		\end{itemize}
\end{theorem}

The mere admissibility of these functions can be used to produce an improved Diophantine inequality for exponential sums for dimensions $N=2,3,4,5$ (improving a result of \cite{BMV}; see also Section \ref{BMV_section}).

\begin{theorem}\label{BMV_improved_thm}
Let $F(\bx)$ be a $\nu(N)$-admissible function satisfying $\ft F(\bo 0)>0$. Let $\ep_n\in (0,1/2]$ for $n=1,...,N$ and let $\bxi_{m}=(\xi_{m,1},...,\xi_{m,N})\in \R^N/\Z^N$ be vectors for $m=1,...,M$  such that 
$$
\max_{n=1,...,N} \frac{\| \xi_{m,n}\|}{\ep_n}\geq 1
$$
for each $m=1,...,M$. Let $\widetilde \LL = \{\bl \in \Z^N : |\ep_n\ell_n| < 1, \ n=1,...,N\}$. Then 
		\[
			\frac{\ft F(\bo 0)}{\|\ft F\|_\infty}M \leq \underset{{\bn \neq \bo 0}}{\dsum_{\bn\in \widetilde \LL}} \left| \dsum_{m=1}^{M}e(\bn\cdot\bxi_{m})  \right|.
		\]
In particular, due to the constructions on Section \ref{conMin}, if the dimension $N=1,2,3,4$ or $5$ we have
	\begin{equation}\label{BMV_imporoved_ineq}
c_NM \leq \underset{{\bn \neq \bo 0}}{\dsum_{\bn\in \widetilde \LL}} \left| \dsum_{m=1}^{M}e(\bn\cdot\bxi_{m})  \right|,
		\end{equation}
		with $c_N$ depending only on the dimension.
\end{theorem}

\section*{Acknowledgements} We thank Enrico Bombieri, Emanuel Carneiro, Arie Israel, Jeffrey Lagarias, Victor Miller, Hugh Montgomery, Jeffrey Vaaler, Robert Vanderbei, and Rachel Ward for helpful feedback and encouragement. The last named author would like to acknowledge support from NSF grants DMS-0943832 and DMS-1045119. The third named author acknowledges the support from University of Alberta Startup funds and the SFB 1060 grant from the Hausdorff Center for Mathematics.

\section{Preliminaries}
In this section we prove some crucial results and recall as well some basic facts about the theory of Paley-Wiener spaces and extremal functions.

For a given function $F:\Rn \ra \R$ we define its Fourier transform as
	\est{
		\tF(\bxi) = \dint_{\Rn} F(\bx) e^{-2\pi i \bx \cdot \bxi}d\bx.
	}
In this paper we will almost always deal with functions $F(\bx)$ that are integrable and whose Fourier transforms are supported in the box 
	\[
		Q_{N}=[-1,1]^N.
	\]
For this reason, given $p\in [1,2]$ we define $PW^p(Q_N)$ as the set of functions $F\in L^p(\R^N)$ such that their Fourier transform is supported in $Q_N$. By Fourier inversion these functions can be identified with analytic functions that extend to $\C^N$ as entire functions. The following is a special case of Stein's generalization of the Paley-Wiener theorem.

\begin{theorem}[Stein, \cite{SW}]\label{exptypethm}
Let $p\in[1,2]$ and $F\in L^p(\R^N)$. The following statements are equivalent:
\begin{enumerate}[$(i)$]
\item ${\rm supp \,}(\ft F)\subset Q_N$;
\item $F(\bx)$ is a restriction to $\R^N$ of an analytic function defined in $\C^N$ with the property that there exists a constant $C>0$ such that  
	\begin{equation*}
		|F(\bx+i\by)| \leq C \exp\left[2\pi \sum_{n=1}^N|y_n|\right]
	\end{equation*}
	for all $\bx,\by\in\R^N$.
\end{enumerate}
\end{theorem}

\noindent {\bf Remark.} In particular this theorem implies that every function $F\in PW^1(Q_N)$ is bounded on $\R^N$, hence $PW^1(Q_N)\subset PW^2(Q_N)$.
\smallskip

\begin{theorem}[P\'olya-Plancherel, \cite{PP}]\label{PPthm}
 If $\bxi_1,\bxi_2,...$ is a sequence in $\R^N$ satisfying that $\|\bxi_n-\bxi_m\|_{\infty}\geq \ep$ for all $m\neq n$ for some $\ep>0$ then
\begin{equation*}
\sum_{n}|F(\bxi_n)|^p \leq C(p,\ep)\int_{\R^N}|F(\bxi)|^p d\bxi
\end{equation*}
for every $F\in PW^p(Q_N)$.
\end{theorem}

\begin{proposition}[Poisson Summation for $PW^1(Q_N)$]\label{prop:Poisson}
For all $F\in PW^1(Q_N)$ and for any $\bt\in\R^N$ we have 
\begin{equation}\label{poisson-sum}
\int_{\R^N}F(\bx)d\bx = \sum_{\bn\in\Z^N} F(\bn+\bt),
\end{equation}
where the convergence is absolute and uniform on compact subsets of $\bt \in \R^N$.
\end{proposition}

Let $F\in PW^2(Q_N)$. If $\bt\in \C^{N-k}$, then the function $\by\in\R^k\mapsto G_{\bt}(\by)=F(\by,{\bo t})$ is the inverse Fourier transform of the following function
$$
\bxi\in\R^k\mapsto\int_{Q_{N-k}}\ft F(\bxi, {\bo u})e^{2\pi i \bt\cdot\bo u}d{\bo u}.
$$
Since $\ft F\in L^2(\R^N)$, we conclude that the above function has finite $L^2(\R^k)$-norm and as a consequence $G_{\bt}\in PW^2(Q_k)$. A similar result is valid for $p=1$, but only for $\nu(N)$-admissible functions.

\begin{lemma}\label{slice}
Let $N>k>0$ be integers. If $F(\bx)$ is $\nu(N)$-admissible then the function $\by\in\R^k\mapsto F(\by, {\bo 0})$ with ${\bo 0}\in\R^{N-k}$ is $\nu(k)$-admissible and
$$
\int_{\R^N}F(\bx)d\bx \leq \int_{\R^k} F(\by,{\bo 0})d\by.
$$
\end{lemma}
\begin{proof}
We give a proof only for the case $N=2$ since it will be clear that the general case follows by an adaption of the following argument. 

Let $F(x,y)$ be a function admissible for $\nu(2)$ and define $G(x)=F(x,0)$. By Fourier inversion we obtain that
$$
G(x)=\int_{-1}^1\left( \int_{-1}^1\ft F(s,t)dt \right)e^{2\pi i sx}ds.
$$
This shows that $G\in PW^2(Q_1)$. Now, for every $a\in(0,1)$ define the functions
$$
G_a(x)=G((1-a)x)\left(\frac{\sin(a\pi x)}{a\pi x}\right)^2.
$$
and
$$
F_a(x,y)=F((1-a)x,y)\left(\frac{\sin(a\pi x)}{a\pi x}\right)^2.
$$
By an application of Holder's inequality and Theorem \ref{exptypethm}, we deduce that $G_a\in PW^1(Q_1)$ and $F_a\in PW^1(Q_2)$ for all $a\in(0,1)$. Hence, we can apply Poisson summation to conclude that
\begin{align*}
\int_{\R}G_a(x)dx = & \sum_{n\in\Z} G((1-a)n)\left(\frac{\sin(a\pi n)}{a\pi n}\right)^2 \\ \geq & \sum_{(n,m)\in\Z^2} F((1-a)n,m)\left(\frac{\sin(a\pi n)}{a\pi n}\right)^2 \\
= & \int_{\R^2}F((1-a)x,y)\left(\frac{\sin(a\pi x)}{a\pi x}\right)^2 dxdy,
\end{align*}
where the above inequality is valid because the function  $F(x,y)$ is a minorant of the box $Q_2$. Observing that $G_a(x) \leq {\bo 1}_{Q_1/(1-a)}(x)$ for every $x\in\R$, we  can apply Fatou's lemma to conclude that
\begin{align*}
\int_{\R}[{\bo 1}_{Q_1}(x)-G(x)]dx \leq & \liminf_{a\to 0} \int_{\R}[{\bo 1}_{Q_1/(1-a)}(x)-G_a(x)]dx \\
\leq & \int_{\R}{\bo 1}_{Q_1}(x)dx - \limsup_{a\to 0} \int_{\R^2}F((1-a)x,y)\left[\frac{\sin(a\pi x)}{a\pi x}\right]^2 dxdy \\
= & \int_{\R}{\bo 1}_{Q_1}(x)dx - \int_{\R^2} F(x,y)dxdy < \infty.
\end{align*}
This concludes the proof.
\end{proof}

We now introduce an interpolation theorem which has proven indispensable throughout our investigations.

\begin{proposition}
For all $F\in PW^2(Q_N)$ we have
	\es{\label{int-form-gen}
		F(\bo x) = \dprod_{n=1}^N \bigg( \dfrac{\sin\pi x_{n}}{\pi} \bigg)^{2}  \sum_{\bo n\in\Z^N} \sum_{\,\,\bo j\in\{0,1\}^N}\frac{\partial_{\bo j}F(\bo n)}{(\bo x-\bo n)^{\bo 2-\bo j}}
}
where $\partial_{\bo j}=\partial_{j_1...j_N}$ and $(\bo x-\bo n)^{\bo 2-\bo j}=(x_1-n_1)^{2-j_1}...(x_N-n_N)^{2-j_N}$ and the right hand side of \eqref{int-form-gen} converges uniformly on compact subsets of $\Rn$.
\end{proposition}
\begin{proof}
This proposition is proven by induction using Vaaler's result \cite[Theorem 9]{V} as the base case and Theorem \ref{PPthm} (P\'olya-Plancherel), which  guarantees that the sequence $\{F(\bn) : \bn\in\Z^N\}$ is square summable for any $F\in PW^2(Q_N)$. Also note that by Fourier inversion $PW^2(Q_N)$ is closed under partial differentiation.
\end{proof}

Finally, the next lemma demonstrates that extremal functions always exist for $\nu(N)$ and other minorization problems.

\begin{lemma}\label{sequenceLemma}
Suppose $G\in L^{1}(\Rn)$ is a real valued function. Let $F_{1}(\bx), F_{2}(\bx),...$ be a sequence in $PW^1(Q_N)$ such that $F_{\ell}(\bx)\leq G(\bx)$ for each $\bx\in \Rn$ and each $\ell$. Assume that there exists $A>0$ such that $\ft{F}_{\ell}({\bo 0})\geq -A$ for each $\ell$. Then there exists a subsequence $F_{\ell_{k}}(\bx)$ and a function $F\in PW^1(Q_N)$ such that $F_{\ell_{k}}(\bx)$ converges to $F(\bx)$ uniformly on compact sets as $k$ tends to infinity. In particular, we deduce that $F(\bx)\leq G(\bx)$ for each $\bx\in \Rn$ and $\limsup_{k\rightarrow \infty} \ft{F}_{{\ell_k}}({\bo 0})\leq \ft{F}({\bo 0})$.
\end{lemma}
\begin{proof}
	By the remark after Theorem \ref{exptypethm}, each $F_{\ell}\in PW^2(Q_N)$ and we can bound their $L^2(\R^N)$-norm in the following way
		\[
			\|F_{\ell} \|_{2}=\|\ft{F}_{\ell} \|_2 \leq \mathrm{vol}_{N}(Q_{N})^{1/2} \| \ft{F}_{\ell}\|_{\infty} \leq 2^{N/2} \| F_{\ell} \|_{1}
		\]
	and
		\es{\label{eq-40}
			\| F_{\ell}\|_{1}\leq \| G- F_{\ell}\|_{1}+\|G\|_{1}\leq 2\|G\|_{1}+A.
		}
Hence the sequence $F_{1}(\bx),F_{2}(\bx),...$ is uniformly bounded in $L^{2}(\Rn)$ and, by the Banach-Alaoglu theorem, we may extract a subsequence (that we still denote by $F_{\ell}(\bx)$) that converges weakly to a function $F\in PW^2(Q_N)$. By Theorem \ref{exptypethm} we can assume that $F(\bx)$ is continuous. By using Fourier inversion we have
$$
F_\ell(\bx) = \int_{Q_N} \ft F_\ell(\bxi)e(\bxi\cdot \bx) d\bxi.
$$
Thus, the weak convergence implies that $F_\ell(\bx)\to F(\bx)$ point-wise for all $x\in\R^N$. Fourier inversion also shows that $\|F_\ell\|_{\infty}\leq 2^N \|F_\ell\|_1$. However, we also have
$$
|\partial_j F_\ell(\bx)| =  2\pi \bigg|\int_{Q_N} \xi_j\ft F_\ell(\bxi)e(\bxi\cdot \bx) d\bxi\bigg| \leq 2\pi \|F_\ell\|_1 2^{N}.
$$
We can use \eqref{eq-40} to conclude that $|F_\ell(\bx)|+|\nabla F_\ell(\bx)|$ is uniformly bounded in $\R^N$. We can apply the Ascoli-Arzel\`a theorem to conclude that, by possibly extracting a further subsequence, $F_\ell(\bx)$ converges to $F(\bx)$ uniformly on compact sets of $\R^N$.

We conclude that $G(\bx)\geq F(\bx)$ for each $\bx\in\Rn$. By applying Fatou's lemma to the sequence of functions $G(\bx)-F_{1}(\bx),G(\bx)- F_{2}(\bx),...$ we find that $F\in L^{1}(\Rn)$ and 
	\[
		\displaystyle\limsup_{\ell\ra \infty} \dint_{\Rn} F_\ell(\bx)d\bx \leq \dint_{\Rn}F(\bx)d\bx.
	\]
	This concludes the lemma.
\end{proof}


\section{Proofs of the Main Results}

The next theorem is the cornerstone in the proof of our main results. This theorem is in stark contrast with the one dimensional case. In the one dimensional case Selberg's function interpolates at all lattice points, and is therefore extremal. In two dimensions, on the other hand, if a minorant interpolates everywhere except for possibly the origin, then it is identically zero. This theorem is therefore troublesome because it seems to disallow the possibility of using interpolation (in conjunction with Poisson summation) to prove an extremality result.

\begin{theorem}\label{latticeInterpolate}
Let $F(x,y)$ be admissible for $\nu(2)$ such that $F(n,m)=0$ for each non-zero $(n,m)\in\Z^2$ and $F(0,0) \geq 0$, then $F(x,y)$ vanishes identically.
\end{theorem}
\begin{proof}
{\it Step 1.} First we assume that the function $F(x,y)$ is invariant under the symmetries of the square, that is,
\es{\label{symm-F}
F(x,y)=F(y,x)=F(|x|,|y|)
}
for all $x,y\in\R$. We claim that for any $(m,n)\in\Z^2$ we have:
\begin{enumerate}[(a)]
\item $\partial_{x}F(m,n)=0$ if $(m,n) \neq (\pm 1,0)$ and $\partial_{y}F(m,n)=0$ if $(m,n)\neq (0,\pm 1)$;
\item $\partial_{xx}F(m,n)=0$ if $n\neq 0$ and $\partial_{yy}F(m,n)=0$ if $m\neq 0$;
\item $\partial_{xy}F(m,n)=0$ if $n \neq \pm 1$ or $m\neq \pm 1$.
\end{enumerate}
 
We can apply Theorem \ref{slice} to deduce that, for each fixed non-zero integer $n$, the function $x\in\R\mapsto F(x,n)$ is a non-positive function belonging to $PW^1(Q_1)$ that vanishes in the integers, hence identically zero by formula \eqref{int-form-gen}. Also note that the points $(m,0)$ for $m\in\Z$ with $|m|> 1$ are local maximums of the function $x\in\R \mapsto F(x,0)$. These facts in conjunction  with the invariance property \eqref{symm-F} imply items (a) and (b).

Finally, note that a point $(m,n)$ with $|n|>1$ has to be a local maximum of the function $F(x,y)$. Thus, the Hessian determinant of $F(x,y)$ at such a point has to be non-negative, that is,
$$
\mathrm{Hess}_F(m,n) := \partial_{xx}F(m,n)\partial_{yy}F(m,n) - [\partial_{xy}F(m,n)]^2 \geq 0
$$
However, by item (b), $\partial_{xx}F(m,n)=0$ and we conclude that $\partial_{xy}F(m,n)=0$. This proves item (c) after using again the property \eqref{symm-F}.
\smallskip

{\it Step 2.}
We can now apply formula \eqref{int-form-gen} and deduce that $F(x,y)$ has to have the following form
$$
F(x,y)=\bigg(\frac{\sin \pi x}{\pi x}\bigg)^2\bigg(\frac{\sin \pi y}{\pi y}\bigg)^2 \bigg\{F(0,0) +a\frac{x^2}{x^2-1} +a\frac{y^2}{y^2-1}+b\frac{x^2y^2}{(x^2-1)(y^2-1)}\bigg\},
$$
where $a=2\partial_xF(1,0)$ and $b=4\partial_{xy}F(1,1)$. Denote by $B(x,y)$ the expression in the brackets above and note that it should be non-positive if $|x|\geq 1$ or $|y|\geq 1$. We deduce that
$$
F(0,0)+a+(a+b)\frac{x^2}{x^2-1}=B(x,\infty) \leq 0
$$ 
for all real $x$. We conclude that $a+b=0$, $F(0,0)\leq -a$ and
$$
B(x,y) = F(0,0) + a\bigg[1-\frac{1}{(x^2-1)(y^2-1)}\bigg].
$$
For each $t>0$, the set of points $(x,y)\in\R^2\setminus Q_2$ such that $(x^2-1)(y^2-1)=1/t$ is non-empty and $B(x,y)=F(0,0)+a - at$ at such a point. Therefore $a\geq 0$ and we deduce that $F(0,0)\leq 0$. We conclude that $F(0,0)=0$, which in turn implies that $a=0$. Thus $F(x,y)$ vanishes identically.
\smallskip

{\it Step 3.} Now we finish the proof. Let $F(x,y)$ be a $\nu(2)$-admissible function such that $F(0,0) = \ft F(0,0) \geq 0$. Define the function
$$
G_1(x,y)=\frac{F(x,y)+F(-x,y)+F(x,-y)+F(-x,-y)}{4}.
$$
Clearly the following function
$$
G_0(x,y)=\frac{G_1(x,y)+G_1(y,x)}{2}
$$
is also $\nu(2)$-admissible and $G_0(0,0)=\ft G_0(0,0)=F(0,0) \geq 0$. Moreover, $G_0(x,y)$ satisfies the symmetry property \eqref{symm-F}.
By Steps $1$ and $2$ the function $G_0(x,y)$ must vanish identically. Thus, we obtain that
$$
G_1(x,y)=-G_1(y,x).
$$
However, since $G_1(x,y)$ is also $\nu(2)$-admissible we conclude that $G_1(x,y)$ is identically zero outside the box $Q_2$, hence it vanishes identically. An analogous argument can be applied to the function $G_2(x,y)=[F(x,y)+F(-x,y)]/2$ to conclude that this function is identically zero outside the box $Q_2$, hence it vanishes identically. Using the same procedure again we finally conclude that $F(x,y)$ vanishes identically and the proof of the theorem is complete.
\end{proof} 

\subsection{Proof of Theorem \ref{latticeInterpolateN}}
The proof is done via induction and the base case is Theorem \ref{latticeInterpolate}. Assume that the theorem is proven for some dimension $N\geq 2$. Let $F(\bx,x_{N+1})$ be a $\nu(N+1)$-admissible function such that $F(\bn,m)=0$ for all non-zero $(\bn,m)\in \Z^{N+1}$. Now, for every fixed $t\in\R$ define $G_t(\bx)=F(\bx,t)$. An application of Lemma \ref{slice} shows that $G_t\in PW^1(Q_N)$ for all $t\in\R$ and is $\nu(N)$-admissible if $|t|<1$ and non-positive if $|t|\geq 1$. Moreover, for any fixed non-zero $m\in\Z$ we have $G_m(\bn)=0$ for all $\bn\in\Z^N$, thus by induction we have $G_m\equiv 0$ for every non-zero $m\in\Z$. By symmetry we have $F(\bx,x_{N+1})=0$ if one of its entries is a non-zero integer. We conclude that the $\nu(N)$-admissible function $G_t(\bx)$ satisfies $G_t(\bn)=0$ for every non-zero $\bn\in\Z^N$. By induction again, $G_t\equiv 0$ for all real $t$. This implies that $F\equiv0$ and this finishes the proof.

\subsection{Proof of Theorem \ref{nuTheorem}}
The item $(i)$ is a direct consequence of Lemma \ref{sequenceLemma} while item $(iii)$ is a consequence of Theorem \ref{Deltathm} item $(iv)$. It remains to show item $(ii)$. 

Clearly by Lemma \ref{slice}, we have $\nu(N)\geq \nu(N+1)$. Suppose by contradiction that $\nu(N)=\nu(N+1)$. Let $(\bx,t)\in\R^N\times \R\mapsto F(\bx,t)$ be an extremal function for $\nu(N+1)$. Let $G_{m}(\bx)=F(\bx,m)$ for each $m\in \Z$. Lemma \ref{slice} implies that $G_m(\bx)$ is also admissible for $\nu(N)$ (if $m\neq 0$ then the function is non-positive). By the Poisson summation formula we have for each non-zero $m\in\Z$
		\begin{equation}\label{eq:sum1}
			\ft{F}({\bo 0})=\dsum_{\bn\in\Z^N}\dsum_{k\in\Z} F(\bn,k) \leq \dsum_{\bn\in\Z^N}( F(\bn,m) + F(\bn,0)) = \ft{G}_{m}(0) + \ft{G}_{0}(0).
		\end{equation}
By assumption 
		\begin{equation}\label{eq:assum1}
			 \ft{G}_{0}(0)\leq \nu(N)=\nu(N+1)=\ft{F}({\bo 0})
		\end{equation}
Combining \eqref{eq:sum1} and \eqref{eq:assum1} yields  $0 \leq  \ft{G}_{m}(0) $ for each $m\neq0$. However, the function $G_{m}(\bx)\leq 0$ for each $\bx\in\Rn$ whenever $m$ is a non-zero integer. Consequently,  $G_{m}(\bx)$ vanishes identically. It follows that $F(\bn)=0$ for each non-zero $\bn\in\Z^{N+1}$. By Theorem \ref{latticeInterpolateN}, $F(\bx)$ vanishes identically. Therefore $\nu(N+1)=\nu(N)=0$, a contradiction.  The theorem is finished.

\subsection{Proof of Theorem \ref{Deltathm}}
First we prove item $(i)$. Assume by contradiction that there exists a sequence of $\Delta(N)$-admissible functions $F_\ell(\bx)$ $\ell=1,2,...$ such that $M_\ell=\max_{x\in Q_N} \{F_\ell(\bx)\}$ converges to $\infty$ when $\ell\to \infty$. Let $G_\ell(\bx)=F_\ell(\bx)/M_\ell$, and note that $G_\ell(\bx)$ is $\nu(N)$-admissible for all $\ell$. Also let $\bx_\ell\in Q_N$ be such that $F_\ell(\bx_\ell)=M_\ell$. We can assume by compactness that $\bx_\ell\to \bx_0$. By Lemma \ref{sequenceLemma} we may also assume that there exists a function $G(\bx)$, $\nu(N)$-admissible such that $G_\ell(\bx)$ converges uniformly on compact sets to $G(\bx)$. We also have by Lemma \ref{sequenceLemma} that
$$
0 \leq \limsup_\ell \int_{\R^N} G_\ell(\bx)d\bx \leq \int_{\R^N} G(\bx)d\bx.
$$
However, $G_\ell(\bo 0)=1/{M_\ell}\to 0$ and thus $G(\bo 0)=0$. By the Poisson summation formula, for any fixed non-zero $\bn \in\Z^N$ we have
$$
0\leq \ft G_\ell(\bo 0) \leq 1/M_\ell + G_\ell(\bn).
$$
Thus, we conclude that $G_\ell(\bn)\to 0$ as $\ell \to \infty$. This, implies that $G(\bn)=0$ for all $\bn\in\Z^N$. By Theorem \ref{latticeInterpolateN} we conclude that $G(\bx)$ vanishes identically. However, by uniform convergence we have $G(\bx_0)=1$, a contradiction. This proves item $(i)$

Item $(ii)$ is a consequence of Lemma \ref{sequenceLemma} in conjunction with item $(i)$. Item $(iii)$ can be proven exactly as in Theorem \ref{nuTheorem} item $(ii)$, since now we know that extremizers exist. It remains to show the upper bound of item $(iv)$. For this we will show a stronger result.

\begin{lemma}\label{Delta-dec-lemma}
The function 
$$
\delta(N)=\frac{1-\Delta(N)}{N}
$$
is non-decreasing. That is, if $\Delta(N)>0$ and $M<N$ then $\delta(M)\leq \delta(N)$.
\end{lemma}

\begin{proof}
For a given $\bn \in \Z^N$ let $\sigma(\bn)$ denote the quantity of distinct numbers in $\Z^N$ that can be constructed by only permuting the entries in $\bn$. It is simple to see that if $M<N$, $\bm\in \Z^M$ is non-zero and $(\bm, \bo 0) \in\Z^N$ then
$$
\sigma(\bm, \bo 0) \geq (N/M) \sigma(\bm),
$$
and equality is attained if $\bm$ has only one entry different than zero. Let $\Gamma^N$ be the subset of { non-zero} $\bn =(n_1,...,n_N)\in \Z_+^N$ such that $n_1\geq n_2 \geq ...\geq n_N\geq 0$ ($\Z_+=\{0,1,2,...\}$). Note that $(\Gamma^M,\bo 0)\subset \Gamma^N$ if $M<N$. Also let $\ep(\bn)$ be the number of non-zero entries in a vector $\bn\in\Gamma^N$. Now, for a given $N$, let $F_N(\bx)$ be an extremal function for the $\Delta(N)$ problem. We can assume that it is invariant under the symmetries of $Q_N$. Define $G_N(\by)=F_N(\by,\bo 0)$ for every $\by\in\R^M$, $M<N$. By Poisson summation we obtain
\est{
\Delta(N) =\ft F_N(\bo 0) & = 1+ \sum_{\bn\in\Gamma^N} 2^{\ep(\bn)}\sigma(\bn)F_N(\bn) \\ & \leq 1+ \sum_{\bm\in\Gamma^M} 2^{\ep(\bm,\bo 0)}\sigma(\bm,\bo 0)F_N(\bm,\bo 0) \\ & = 1 +  \sum_{\bm\in\Gamma^M} 2^{\ep(\bm)}\sigma(\bm,\bo 0) G_N(\bm) \\
& \leq 1 + (N/M)\sum_{\bm\in\Gamma^M} 2^{\ep(\bm)}\sigma(\bm) G_N(\bm) = 1 + (N/M)(\ft G_N(0)-1) \\ & \leq 1 + (N/M)(\Delta(M)-1).
}
We conclude that
$$
\frac{1-\Delta(N)}{N} \geq \frac{1-\Delta(M)}{M},
$$
and this finishes the lemma.
\end{proof}
\begin{proof}[Proof of Theorem \ref{Deltathm} continued]
The previous lemma implies that if $\Delta(N)>0$ then
$$
\frac{1}{N} > \frac{1-\Delta(N)}{N}=\delta(N)\geq \delta(k) = \frac{1-\Delta(k)}{k},
$$
for any $k\leq N$. We conclude that
$$
\frac{k}{1-\Delta(k)}>N,
$$
and this finishes the proof of the theorem.
\end{proof}
\smallskip

\noindent {\bf Remark.} Let $k<N\leq N_c$, then
$$
\Delta(k) \geq (1-k/N) + (k/N)\Delta(N).
$$
Since the right hand side above is always larger than $\Delta(N)$, this inequality produces better lower bounds for lower dimensions once a lower bound is given for a higher dimension.

\section{Linear Programming Bounds}
In this section we'll solve the linear program in Theorem \ref{nuLPbound} to calculate explicit upper bounds for $\nu(N)$ and $\Delta(N)$ for $N=2,3,4,5$. First, we will exploit the symmetry of the problem to make the linear program in Theorem \ref{nuLPbound} less computationally expensive. We'll then describe our strategy for computing upper bounds via this new linear program. Finally, we describe how to modify this strategy to derive rigorous bounds.

Let $R\in \Z_{>1}$. For $\bo{x}=(x_1,\dots,x_N)\in\R^N$, define $\cos(\bo{x}) := \prod_{i=1}^N \cos(x_i)$. Let $\A := \{ \bo{x} \in \R^N : 0 < x_1 < \dots < x_N < R\}$. Let $\Sigma(Q_N)$ denote the group of symmetries of the unit cube. Note that this group has a natural action on $\R^N$ of permuting coordinates and switching signs. The orbit of any point in $\A$ under $\Sigma(Q_N)$ has cardinality $2^N N!$. Let $S_N$ be the symmetric group on $N$ elements. For $\sigma\in S_N$ and $\bo x = \left( x_1, \dots, x_N\right) \in \R^N$ we write $\sigma(\bo x):= \left(x_{\sigma(1)}, \dots,x_{\sigma(N)}\right)$, i.e. we let $\sigma$ act on the indices of the components of $\bo x$. For $x,y\in\R^N$ we write $xy$ to mean the component-wise product of $x$ and $y$; we'll use $\langle x , y\rangle$ to mean the scalar product.

Note that we solved the linear program with CVX (using the Gurobi solver) in MATLAB. The rational arithmetic was done in Maple. The code will be made available before publication on the first author's website.

\subsection{Reducing the size of the linear program} 
Recall that Theorem \ref{nuLPbound} gives a linear program in which we want to optimize a set of weights $\omega_1,\dots,\omega_L$, where each weight $\omega_i$ corresponds to shifting the periodic summation of the extremal function $F$ by a single point $\bo y_i$ in $\R^N$. If instead we let each weight $\omega_i$ correspond to the shifts by each of the points in $\Sigma(\bo y_i):= \{ \bo z : \sigma(\bo y_i) = z \text{ for some } \sigma \in \Sigma(Q_n)\}$, then we can exploit the symmetry of the problem to get the following simplification of Theorem \ref{nuLPbound}.

\begin{theorem}\label{cosLP}
Let $\bo{y}_0 = 0\in \R^N$ and suppose $\bo{y}_1,\dots, \bo{y}_L \in \A$, $\omega_0,\dots,\omega_L \in \R^{\ge 0}$ are such that
\begin{enumerate}
\item{$\omega_0 + \sum_{i=1}^L \omega_i 2^N \sum_{\sigma\in S_N}\cos\left(\frac{2\pi \sigma\left(\bo y_i\right) \bo{n}}{N}\right) =0$ for all $n \in \A \cap \Z^N$ such that $0 < || n||_\infty < R$}
\item{$\omega_0 + \sum_{i=1}^L \omega_i 2^N N!= R^N$}
\end{enumerate}
Then
\[
\nu(k) \le \omega_0 + \underset{\{1\le i \le L : ||\bo{y}_i||_\infty< 1\}} \sum \omega_i 2^N N!
\]
and, if $||y_i||_\infty \ge 1$ for all $ 0<i\le L$, then
\[
\Delta(k) \le \omega_0
\]
\end{theorem}

\subsection{A simple algorithm}
In order to use Theorem \ref{cosLP} to compute explicit bounds, we first fix values of $N$ and $R$ and generate a large number of random points $\bo{y}_i \in \A$. Then we solve the linear program in Theorem \ref{cosLP}, store the values of $\bo{y}_i$ for which $w_i > 0$ (solutions are very sparse due to the relatively small number of constraints), generate a large number of new points and add them to the collection of values, and then repeat this process until the upper bound appears to stabilize.

For larger values of $N$ and $R$ sometimes this method is not good at finding a feasible value of $w$ until $L$ is very large. In this case one can speed things up by first solving the problem by taking the $\bo{y}$'s to belong to the lattice $(1/S)\Z^N$ for some value of $S>R$; in our experience this always gives a feasible value of $w$ and then the bound can be improved by remembering the nonzero entries of $w$ and the corresponding values of $\bo y$, generating a random set of $\bo y$'s, and iterating as above.

In our experience when we solve the linear program all of the nonzero values of $w_i$ correspond to points $\bo y_i$ which satisfy $|| y_i ||_\infty > 1$. So, in our experience, this method gives the same bounds on $\Delta(N)$ and $\nu(N)$.

We summarize the upper bounds for $\Delta(N)$ in the following table.

 \begin{table}[h!]
 \label{table:data}
 \title{\bf Upper bounds for $\Delta(N)$.}
 \bigskip
 
\begin{tabular}{|c|}
\hline
0.9946333 ($N=2$, $R=16$, $S=5$) \\ \hline 0.9849928 ($N=3$, $R=10$, $S=5$) \\ \hline 0.9802947 ($N=4$, $R=5$, $S=6$) \\ \hline 0.9553936 ($N=5$, $R=7$, $S=4$) \\ \hline
\end{tabular}
\bigskip
\end{table}

\subsection{Making the bounds rigorous}
Note that the bounds in the previous section were obtained using floating point arithmetic and therefore are not rigorous (the upper bound $\Delta(5)\leq 0.9553936$, if correct within $2$ significant digits, would produce $N_c\leq 125$). In this section we will describe how to use rational arithmetic to remedy this. We carry out the computations only for the case $N=2$, $R=6$ and leave the other cases for future work. In this case we are able to get a rigorous upper bound of \footnote{In fact, we get a rigorous upper bound of a rational number slightly smaller than this, but it has too many digits to fit in this paper. The interested reader can find it in the Maple script on the first author's website.}
$$
\Delta(2)\leq 0.997212,
$$
and  this gives a rigorous upper bound of 
$$
N_c \leq 717.
$$
Notice that the linear program in Theorem \ref{cosLP} can be made rational if we choose values of $\bo y_i$ such that $\cos\left(\frac{2\pi n \bo y_i}{R}\right)$ is rational for all values of $n\in \Z^N$. This will happen if and only if $C_i:= \cos\left(\frac{2\pi \bo{y_i}}{R}\right)$ and $S_i:=\sin\left(\frac{2\pi\bo{y_i}}{R}\right)$ are both rational numbers. Moreover, since we don't actually need the values of $\bo y_i$ to solve the linear program, just the values of $\cos$, instead of generating random points $\bo y_i$ we can generate random rational points on the circle, or equivalently random Pythagorean triples. These will be the values of $S_i$ and $C_i$.

Rather than solve the linear program using rational arithmetic, which is computationally expensive, we solve the problem using floating point arithmetic to identify the nonzero entries of $w_i$ and then use rational arithmetic to solve the resulting linear system. In our experience the number of nonzero entries is always the same as the number of equality constraints. Therefore the only rational arithmetic we have to do is solving a relatively small full rank square linear system.

\section{Explicit Minorants in Low dimensions}\label{conMin}
We define an auxiliary variational quantity $\lambda(N)$ over a more restrictive set of admissible functions than $\nu(N)$. Let 
	\[
		\lambda(N)=\sup \dint_{\Rn}F(\bx)d\bx
	\]	
where the supremum is taken over functions $F(\bx)$ that are admissible for $\nu(N)$ and, in addition, $F({\bo 0})=1$, and 
	\[
		F(\bn)=0
	\]
for each non-zero $\bn\in\Z^N$ unless $\bn$ is a ``corner" of the box $Q_{N}$. Here, a corner of the box $Q_N$ is a vector $\bn\in \partial Q_N\cap \Z^N$ with at least $2$ non-zero entries. This definition makes any $k$-dimensional slice of an admissible function for $\lambda(N)$ ($k<N$) admissible for $\lambda(k)$, which in turn implies that
\[
\lambda(N+1)\leq \lambda(N)
\]
for all $N$. We note that Selberg's functions (see Appendix) are always admissible for $\lambda(N)$ but have negative integral. Our aim is to mimick Selberg's construction but to incorporate a correction term so that our minorants have positive integral. Notice that by Theorem \ref{nuTheorem} it is impossible to do this in sufficiently high dimensions.

Making use of the interpolation formula \eqref{int-form-gen} we conclude that every function $F(\bx)$ admissible for $\lambda(N)$ has the following useful representation
\begin{equation}\label{simple-form}
F(\bx) = S(\bx) P(\bx)
\end{equation}
where
$$
S(\bx) = \prod_{n=1}^N \bigg(\frac{\sin(\pi x_n)}{\pi x_n(x_n^2-1)}\bigg)^2
$$
and $P(\bx)$ is a polynomial such that each variable $x_n$ appearing in its expression has an exponent not greater than $4$. Notice that, by Poisson summation, if $F(\bx)$ is admissible for $\lambda(N)$ and is invariant under the symmetries of $Q_N$ then 
	\begin{equation}\label{cornerIntegral}
		\dint_{\Rn}F(\bx)d\bx= 1+ \sum_{k=2}^N \binom N k 2^{k}P(\bo u_k)
	\end{equation}  
where $ \bo u_k=(\overbrace{1,1,1,..,1}^{\text{k times}},0,...,0)$. 

In what follows it will be useful to use a particular family of symmetric functions. For given integers $N\geq k \geq 1$ we  define
\begin{equation*}
\sigma_{N,k}(\bx) = \sum_{1\leq n_1<n_2<...<n_k\leq N} x_{n_1}^2x_{n_2}^2...x_{n_k}^2
\end{equation*}
and
\begin{equation*}
\wt \sigma_{N,k}(\bx) = \sum_{1\leq n_1<n_2<...<n_k\leq N} x_{n_1}^4x_{n_2}^4...x_{n_k}^4.
\end{equation*}

\begin{lemma}\label{general_construction_lowerbound}
Let $F_N : \R^N\to \R$ be $\lambda(N)$-admissible function constructed using \eqref{simple-form} with 
$$
P_N(\bx) = \prod_{n=1}^{N} (1-x_n^2) - \sum_{k=1}^N a_k\sigma_{N,k}(\bx) - \sum_{k=1}^N b_k\wt \sigma_{N,k}(\bx)
$$
and $a_k,b_k\geq 0$ for $k=1,...,N$. Let $\by$ denote vectors in $\R^{N+1}$. If $N$ is even then the function $F_{N+1}:\R^{N+1}\to \R$ constructed using \eqref{simple-form} with
$$
P_{N+1}(\by) = \prod_{n=1}^{N+1} (1-y_n^2) - \sum_{k=1}^N a_k\sigma_{N+1,k}(\by) - \sum_{k=1}^N b_k\wt \sigma_{N+1,k}(\by)
$$
is $\lambda(N+1)$-admissible. If $N$ is odd then the function $F_{N+1}:\R^{N+1}\to \R$ constructed using \eqref{simple-form} with
$$
P_{N+1}(\by) = \prod_{n=1}^{N+1} (1-y_n^2) - \sum_{k=1}^N a_k\sigma_{N+1,k}(\by) - \sum_{k=1}^N b_k\wt \sigma_{N+1,k}(\by) - \delta\sigma_{N+1,N+1}(\by),
$$
where $\delta\geq 0$ and
$$
\delta\geq \sup_{|y_n|\geq 1} \frac{ \prod_{n=1}^{N+1} (y_n^2-1) - \sum_{k=1}^N a_k\sigma_{N+1,k}(\by) - \sum_{k=1}^N b_k\wt \sigma_{N+1,k}(\by)}{\sigma_{N+1,N+1}(\by)}
$$
is a $\lambda(N+1)$-admissible function.
\end{lemma}

\begin{proof}
Assume $N$ is even. If all $y_1,...,y_{N+1}$ have moduli less than $1$ then trivially  $P_{N+1}(\by)\leq 1$. If an even number of the variables $y_1,...,y_{N+1}$ have moduli less than $1$, then also clearly $P_{N+1}(\by)\leq 0$.  If an odd number, but not all, have moduli less than $1$, assume for instance $|y_1|< 1$, then we would have  $P_{N+1}(\by)\leq P_{N}(y_2,...,y_{N+1}) \leq 0$.

Assume now that $N$ is odd. Again, if all $y_1,...,y_{N+1}$ have moduli less than $1$ then trivially  $P_{N+1}(\by)\leq 1$. If an odd number of the variables $y_1,...,y_{N+1}$ have moduli less than $1$, then also clearly $P_{N+1}(\by)\leq 0$. If an even number of variables, not non of them, have moduli less than $1$, say $|y_1|<1$, then $P_{N+1}(\by) \leq P_N(y_2,...,y_{N+1}) \leq 0$. If all variables have moduli greater than $1$ then by the choice of $\delta$ we have $P_{N+1}(\by)\leq 0$.
\end{proof}

Note that $\delta=1$ always work, but that is often not the best choice since we want to minimize $\delta$ so to make $\ft F_{N+1}(\bo 0)$ as large as possible, hence this forces the coefficients $b_k$ being not too small. Also note that in this way $P_N(\bx)=P_{N+1}(\bx,0)$. Using the above lemma we were able to construct admissible functions up to dimension $N=5$ by starting with a good two dimensional minorant.

\begin{theorem}\label{construction-lower-bound}
Define the functions $\F_2(x_1,x_2)$, $\F_3(x_1,x_2,x_3)$, $\F_4(x_1,...,x_4)$ and $\F_5(x_1,...,x_5)$ by using representation \eqref{simple-form} and the following polynomials respectively :
\begin{itemize}
\item $P_2(x_1,x_2)=(1-x_{1}^2)(1-x_{2}^2)-\tfrac{1}{16}\wt \sigma_{2,2}(x_1,x_2)$
\smallskip

\item $P_3(x_1,x_2,x_3)= \dprod_{n=1}^{3}(1-x_{n}^2)- \tfrac{1}{16}\wt \sigma_{3,2}(x_1,x_2,x_3)$
\smallskip

\item $P_4(x_1,...,x_4)= \dprod_{n=1}^{4}(1-x_{n}^2)-\tfrac{3}{4}\sigma_{4,4}(x_1,...,x_4) -  \tfrac{1}{16}\wt \sigma_{4,2}(x_1,...,x_4)$
\smallskip

\item $P_5(x_1,...,x_5)= \dprod_{n=1}^{5}(1-x_{n}^2)-\tfrac{3}{4}\sigma_{5,4}(x_1,...,x_5) -  \tfrac{1}{16}\wt \sigma_{5,2}(x_1,...,x_5).$
\end{itemize}
These functions are admissible for $\lambda(2)$, $\lambda(3)$, $\lambda(4)$ and $\lambda(5)$ respectively and their respective integrals are equal to: $63/64 = 0.984375$, $119/128=0.9296875$, $95/128=0.7421975$ and $31/256=0.12109375$.

\end{theorem}

\begin{proof}
The integrals of these functions can be easily calculated using formula \eqref{cornerIntegral}, we prove only their admissibility. We start with $\F_2(\bx)$. Clearly, if $|x_1|>1>|x_2|$ then $P_2(x_1,x_2)<0$. Also, writing $t=|x_1x_2|$ we obtain
\begin{align*}
P_2(x_1,x_2) & = 1+x_1^2x_2^2-x_1^2-x_2^2-x_1^4x_2^4/16
 \\ & \leq 1+x_1^2x_2^2-2|x_1x_2|-x_1^4x_2^4/16
 \\ & = 1+t^2-2t-t^4/16.
\end{align*}
On the other hand, we have
\es{\label{eq-30}
1+t^2-2t-t^4/16 = (1-t)^2 -t^4/16
}
and
\es{\label{eq-31}
1+t^2-2t-t^4/16 = (t-2)^2(4-4t-t^2)/16.
}
If $|x_1|,|x_2|<1$ then $0\leq t < 1$, and by \eqref{eq-30} we deduce that $P_2(x_1,x_2)<1$. If $|x_1|,|x_2|>1$ then $t>1$, and by \eqref{eq-31} we deduce that $P_2(x_1,x_2)\leq 0$. This proves that $\F_2(\bx)$ is $\lambda(2)$-admissible. Lemma \ref{general_construction_lowerbound} shows that $\F_3(\bx)$ is admissible for $\lambda(3)$.

We now deal with $\F_4(\bx)$, which from the proof of Lemma \eqref{general_construction_lowerbound} we only need to worry when $|x_1|,|x_2|,|x_3|,|x_4|>1$. In this case, suppressing the variables, we have
$$
P_4 = 1-\sigma_{4,1} + \sigma_{4,2} - \sigma_{4,3} + \tfrac{1}{4}\sigma_{4,4} - \tfrac{1}{16}\wt \sigma_{4,2}.
$$
Observing that
$$
\sigma_{4,2} - \sigma_{4,3} \leq x_1^2x_2^2 + x_3^2x_4^2,
$$
we obtain
\begin{align*}
1-\sigma_{4,1} + \sigma_{4,2} - \sigma_{4,3}  - \tfrac{1}{16}\wt \sigma_{4,2} 
\leq  &  -1 + P_2(x_1,x_2)  +  P_2(x_3,x_4) \\ & - \tfrac{1}{16}[x_1^4x_3^4+x_1^4x_4^4+x_2^4x_3^4+x_2^4x_4^4].
\end{align*}
Since $P_2(x_1,x_2)\leq 0$ and $P_2(x_3,x_4)\leq 0$, we deduce that
\begin{align*}
P_4(x_1,...,x_4) & \leq -1 + \tfrac{1}{4}x_1^2x_2^2x_3^2x_4^2 - \tfrac{1}{16}[x_1^4x_3^4+x_1^4x_4^4+x_2^4x_3^4+x_2^4x_4^4]
\\ & \leq -1 + \tfrac{1}{16}(x_1^4+x_2^4)(x_3^4+x_4^4) - \tfrac{1}{16}[x_1^4x_3^4+x_1^4x_4^4+x_2^4x_3^4+x_2^4x_4^4]
\\ & = -1.
\end{align*}
This proves that $\F_4(\bx)$ is admissible for $\lambda(4)$. Lemma \ref{general_construction_lowerbound} shows  that $\F_5(\bx)$ is admissible for $\lambda(5)$.
\end{proof}

\section{An Application to Diophantine Inequalities}\label{BMV_section}
 Let $\|x\|$ be the distance from the real number $x$ to the nearest integer $k$ and let $[x]=k$. The following is \cite[Corollary 2]{BMV}.

\begin{theorem}[Barton-Montgomery-Vaaler]
Let $\ep_n\in (0,1/2]$ for $n=1,...,N$ and let $\bxi_{m}=(\xi_{m,1},...,\xi_{m,N})\in \R^N/\Z^N$ be vectors for $m=1,...,M$  such that 
$$
\max_{n=1,...,N} \frac{\| \xi_{m,n}\|}{\ep_n}\geq 1
$$
for each $m=1,...,M$. Let $\LL = \{\bl \in \Z^N : |\ep \ell_n| < N , \ n=1,...,N\}$. Then 
		\[
			\frac{1}{3}M \leq \underset{{\bn \neq \bo 0}}{\dsum_{\bn\in \LL}} \left| \dsum_{m=1}^{M}e(\bn\cdot\bxi_{m})  \right|
		\]
\end{theorem}

By using the $\nu(N)$-admissible minorants for $N=1,2,3,4,5$ constructed in Section \ref{conMin} we can improve the above estimate by reducing the size of $\LL$. 
	
\begin{proof} [{\bf Proof of Theorem \ref{BMV_improved_thm}}]
If $\bu$ and $\bv$ are two vectors in $\R^N$ we write \\ $\bu \bv = (u_1v_1,u_2v_2,....,u_Nv_N)$ (recall that  we write inner products using a central dot). We also write $\bu /\bv =(u_1/v_1,u_2/v_2,....,u_N/v_N)$ if all entries of $\bv$ are non-zero. Let $\bo \ep = (\ep_1,...,\ep_n)$, let $U=\Z^N+ \prod_{n=1}^N (-\ep_n,\ep_n)$ and define
$$
\Psi(\bx)=\sum_{\bn \in \Z^N} F((\bn + \bx)/{\bo \ep}).
$$
By Poisson summation we have
		\[
			\Psi(\bx)=\ep_1\ep_2...\ep_N \sum_{\bn\in \Z^N} \ft{F}({\bo \ep} \bn)e(\bn\cdot \bx) = \ep_1\ep_2...\ep_N \sum_{\bn\in \wt \LL} \ft{F}({\bo \ep} \bn)e(\bn\cdot \bx) 
		\]
Clearly, $\Psi(\bx)$ is a minorant of the indicator function of  $U$ and thus
		\[
			0=\dsum_{m=1}^{M} {\bo 1}_{U}(\bxi_{m}) \geq \dsum_{m=1}^{M}  \Psi(\bxi_{m}).
		\]
We obtain
		\[
		0\geq 	(\ep_1\ep_2...\ep_N )^{-1}\dsum_{m=1}^{M}  \Psi(\bxi_{m})= \ft{F}(\bo 0)M + \sum_{\substack{\bn \in \wt \LL\\ \bn \neq \bo 0}}\ft{F}(\bo \ep \bn) \dsum_{m=1}^{M}e(\bn\cdot \bxi_{m}).
		\]
	The proof is complete upon rearranging terms and applying the triangle inequality.
\end{proof}

\noindent {\bf Remark.} A perhaps more useful and easy to remember inequality that now holds for $N=1,2,3,4$ or $5$ is
	$$
			c_N M \leq \underset{{0<\| \bn\|_{\infty}\leq \ep^{-1}}}{\dsum_{\bn\in\Z^N}} \left| \dsum_{m=1}^{M}e(\bn\cdot\bxi_{m})  \right|,
$$
		if for some fixed $\ep\in(0,1/2]$ the $\ell^\infty$ distance of each $\bo \xi_m$ to the nearest point in $\Z^N$ is at least $\ep$.  The constant $c_N$ only depends on the dimension $N$.		
	
\begin{conjecture}
		If $F(\bx)$ is extremal for the $\nu(N)$-problem then it satisfies $\ft F(\bo 0)=\|\ft F\|_\infty$. In particular, inequality \eqref{BMV_imporoved_ineq} holds for all $N\leq N_c$, where $N_c$ is the critical dimension for the $\nu(N)$-problem.
		\end{conjecture}
		
		We note that we have verified numerically that the functions defined in Section \ref{conMin} do not have Fourier transform with an absolute maximum at the origin. On the other hand they are probably not extremal. The absolute maximum however is very close to the origin. As an extreme example, we have 
\begin{align*}
& \ft \F_2(x_1,x_2) =  \\
&  \left(1 - x_1 + \frac{\sin(2 \pi x_1)}{2 \pi}\right) \left(1 - x_2 +
    \frac{\sin(2 \pi x_2)}{2 \pi}\right) \\
& \frac{-1}{256\pi^2}\left(-2 \pi (1 - x_1) \cos(2 \pi x_1) +
     \sin(2 \pi x_1)\right) \left(-2 \pi (1 - x_2) \cos(2 \pi x_2) +
     \sin(2 \pi x_2) \right),
\end{align*}
if $0<x_1,x_2<1$. Recall that $\ft \F_2$ is supported in $Q_2$ and is invariant under the symmetries of $Q_2$. This function is non-negative and its maximum is attained at \\ $(x_1,x_2)=(0.050626...,0.050626...)$ with a value of $0,9869...$, which is larger than $\ft \F_2(0,0)=63/64=0.9843...$. This would produce a constant $c_2=0.997381...$.


\section*{Appendix: Selberg and Montomgery's Constructions}

In this appendix we will present the box minorant constructions of Selberg and Montgomery and we will preform some asymptotic analysis on their integrals. In particular we will show in which regimes Selberg's minorant is a better approximate than Montgomery's and visa-versa. The interested readers are encouraged to consult \cite{HKW,S,V} for more on Selberg's functions and \cite{BMV,MR947641} for more on Montgomery's functions. Our treatment is by no means exhaustive.

Both constructions begin with the following entire functions
\[
 K(z)=\left(\frac{\sin\pi z}{\pi z}\right)^2
 \]
and
\[
H(z)=\left\lbrace \frac{\sin^2 \pi z}{\pi^2}\right\rbrace\left(  \sum_{n=-\infty}^{\infty}\frac{\sgn(n)}{(z-n)^2} +\frac{2}{z}\right)
\]
where 
\[
\sgn(x)=\begin{cases}
  1 & \text{ if } x>0 \\
  0 & \text{ if } x=0 \\
  -1 & \text{ if } x<0.
\end{cases}
\]

Let $[a_1,b_1],...,[a_N,b_N]\subset \R$ where $b_n>a_n$ for each $n=1,...,N$, and set $B=\prod [a_i,b_i]$. For each $i=1,...,N$ define
\begin{eqnarray*}
V_i(z)  &=& \tfrac{1}{2}H(z-a_i)+\tfrac{1}{2}H(b_i-z) \\
E_i(z)  &=& \tfrac{1}{2}K(z-a_i)+\tfrac{1}{2}K(b_i-z) \\
C_i(z) &=& V_i(z)+E_i(z) \\
c_i(z)  &=& V_i(z)-E_i(z).
\end{eqnarray*}
The following theorem can be deduced from \cite{HKW,S,V}.
\begin{theorem}[Selberg]\label{SelbergMinorantThm}
The function 
	\[ \mathscr{C}_B(\bx)=-(N-1)\prod_{i=1}^{N}C_{i}(x)+\sum_{n=1}^{N}c_{n}(x)\prod_{m\neq n} C_{m}(x)\]
satisfies:
\begin{enumerate}[(i)]
\item $\ft {\mathscr{C}}_B(\bxi)=0$ for each $\|\bxi\|_{\infty}>1$;
\item $\mathscr{C}_B \leq {\bo 1}_{B}(\bx)$ for each $\bx\in\R^N$; and
\item \begin{eqnarray*}
 	\int_{\R^N}\mathscr{C}_B(\bx) d\bx&=&-(N-1)\prod_{i=1}^{N}(b_i-a_i +1) \\&&+\sum_{n=1}^{N}(b_n-a_n -1)\prod_{m\neq n} (b_m -a_m +1).
	\end{eqnarray*}
\end{enumerate}
\end{theorem}

\begin{corollary}
 Let $B=[-\delta,\delta]^N$. We have
 	\[
		\int_{\R^N}\mathscr{C}_{B}(\bx)d\bx>0
	\]
if and only if 
	\[
		\delta>N-\frac{1}{2}.
	\]
On the other hand, if $N$ is fixed, then 
       \[
               \int_{\R^N}\mathscr{C}_{B}(\bx)d\bx= (2\delta)^{N}-(N-1)(2\delta)^{N-1}+O(\delta^{N-2})
        \]
as $\delta\to \infty$.
\end{corollary}
\begin{proof}
Setting $a_{n}=-\delta$ and $b_{n}=\delta$ it follows from Theorem \ref{SelbergMinorantThm} (iii) that 
\[
	\int_{\R^N}\mathscr{C}_{B}(\bx)d\bx= (2\delta +1)^{N-1}(2\delta-(2N-1)).
\]
This quantity is positive if and only if $2\delta-(2N-1)> 0$, which occurs if and only if $\delta> N-\tfrac{1}{2}$. On the other hand,
\[
 (2\delta +1)^{N-1}(2\delta-(2N-1))=(2\delta)^{N}-(2N-1)(2\delta)^{N-1}+O_{N}(\delta^{N-2})
\]
as $\delta\to \infty$.
\end{proof}

The following theorem can be deduced from \cite{MR947641}.

\begin{theorem}[Montgomery]\label{MontgomeryMinorantThm}
The function 
	\[ \mathscr{G}_B(\bx)=\prod_{i=1}^{N}V_{i}(x) -\prod_{i=1}^{N}(V_{i}(x)+2E_i(x)) +\prod_{i=1}^{N}(V_{i}(x)+E_i(x))  \]
satisfies:
\begin{enumerate}[(i)]
\item $\ft {\mathscr{G}}_B(\bxi)=0$ for each $\|\bxi\|_{\infty}>1$;
\item $\mathscr{G}_B \leq {\bo 1}_{B}(\bx)$ for each $\bx\in\R^N$; and
\item \begin{eqnarray*}
 	\int_{\R^N}\mathscr{G}_B(\bx) d\bx&=& \prod_{n=1}^{N}(b_n-a_n) -  \prod_{n=1}^{N}(b_n-a_n+2)+ \prod_{n=1}^{N}(b_n-a_n+1).
	\end{eqnarray*}
\end{enumerate}
\end{theorem}

\begin{corollary}
 Let $B=[-\delta,\delta]^N$, $\epsilon>0$, and $\phi=(1+\sqrt{5})/2$. We have
 	\[
		\int_{\R^N}\mathscr{G}_{B}(\bx)d\bx<0
	\]
if 
	\[
		\delta< \left( \frac{1}{2\log(\phi)} -\epsilon  \right)N =  \left(1.039...-\epsilon  \right)N
	\]
	and 
 	\[
		\int_{\R^N}\mathscr{G}_{rQ_{N}}(\bx)d\bx>0
	\]
if 
	\[
		\delta>\left( \frac{1}{2\log(\phi)} +\epsilon  \right)N =  \left(1.039...+\epsilon  \right)N
	\]	
when $N$ is sufficiently large. When $N$ is fixed and $\delta\to \infty$ we have
\[
\int_{\R^N}\mathscr{G}_{B}(\bx)d\bx= (2\delta)^{N}-(2\delta)^{N-1}+O(\delta^{N-2}).
\]
\end{corollary}
\begin{proof}
We will only prove the first statement of the corollary since the second statement is straightforward. Setting $a_{n}=-\delta$ and $b_{n}=\delta$ we have by Theorem \ref{MontgomeryMinorantThm}
\[
        \int_{\R^N}\mathscr{G}_{B}(\bx)d\bx = (2\delta)^{N}-(2\delta+2)^{N}+(2\delta +1)^{N}.
\]
Since the right hand side remains positive if we divide by $(2\delta)^{N}$ it suffices to determine when 
\[
        1-\left(1 + \frac{1}{\delta}  \right)^{N}+\left(1 +\frac{1}{2\delta}  \right)^{N}>0.
\]
Setting $\delta=N/c$ for some $c>0$ we find that for large $N$
\[
 1-\left(1 + \frac{1}{\delta}  \right)^{N}+\left(1 +\frac{1}{2\delta}  \right)^{N} \approx 1-e^c+e^{c/2}.
\]
The equation $ 1-e^c+e^{c/2}=0$ has one real solution, namely $c=2\log(\phi)$. The function $c \mapsto  1-e^c+e^{c/2}  $ is a decreasing function at  $c=2\log(\phi)$ so if $c<2\log(\phi)$ is a constant independent of $N$, then for $N$ sufficiently large we have 
\[
1-\left(1 + \frac{1}{\delta}  \right)^{N}+\left(1 +\frac{1}{2\delta}  \right)^{N}>0.
\]
On the other hand, if $c>2\log(\phi)$ then 
\[
1-\left(1 + \frac{1}{\delta}  \right)^{N}+\left(1 +\frac{1}{2\delta}  \right)^{N}<0.
\]
The proof of the first statemnt is complete upon setting $\delta=((2\log(\phi))^{-1}\pm \epsilon)N$. 
\end{proof}

It follows from the above corollaries that Montgomery's minorants are better approximates when $\delta$ is very large compared to $N$, and Selberg's are better when $N$ is large compared to $\delta$.


\bibliographystyle{plain}
\bibliography{box(1).bib}

\end{document}